\documentclass[12pt]{amsart}
\usepackage[utf8]{inputenc}
\usepackage{a4wide}
\usepackage{lmodern}

\usepackage{hyperref}
\usepackage{dsfont}
\usepackage{stmaryrd}
\usepackage{todonotes}
\usepackage{cite}
\usepackage[capitalize]{cleveref}

\usepackage{amsmath,amsfonts,amssymb,amsthm}
\usepackage{bbm}
\usepackage{pgf,tikz,pgfplots}
\pgfplotsset{compat=1.14}
\usepackage{mathrsfs}
\usetikzlibrary{arrows}

\allowdisplaybreaks

\usepackage{amsmath}
\usepackage{amssymb}
\usepackage{amsthm}
\usepackage{amsfonts}

\def\inda{N_1}
\def\indi{N_2}
\def\MM{N}
\def\nn{\ell}

\theoremstyle{plain}
\newtheorem{thm}{Theorem}[section]
\newtheorem{lem}[thm]{Lemma}

\newtheorem{dfn}[thm]{Definition}

\theoremstyle{definition}

\theoremstyle{remark}

\usepackage{xspace}

\newcommand\Es[1]{\mathbb{E}\left[#1\right]}
\renewcommand\Pr[1]{\mathbb{P}\left(#1\right)}

\title[Isomorphisms between random $d$-hypergraphs]{Isomorphisms between random $d$-hypergraphs}
 \author{Théo Lenoir}

\begin{document}
\maketitle

\begin{abstract}
We characterize the size of the largest common induced subgraph of two independent random uniform $d$-hypergraphs of different sizes with $d\geq 3$. More precisely, its distribution is asymptotically concentrated on two points, and we obtain as a consequence a phase transition for the inclusion of the smallest hypergraph in the largest one. This generalizes to uniform random $d$-hypergraphs the results of Chatterjee and Diaconis \cite{Diaco} for uniform random graphs.

Our proofs rely on the first and second moment methods.

\end{abstract}

\section{Introduction}
\subsection{Motivation and main result}

The question of the size of the largest common induced subgraph of two graphs is a well-studied problem in algorithmics \cite{Algo1, Algo2}. The probabilistic counterpart of this problem is to identify the distribution of the size of the largest common induced subgraph of two random graphs. It was raised by Chatterjee and Diaconis to understand the following seeming paradox. For every $p\in(0,1)$ the Erd\H{o}s-Renyi graph $G(\infty,p)$ is almost surely isomorphic to the Rado graph \cite{Cameron1997}, thus two $G(\infty,p)$ are almost surely isomorphic. However, the probability of having two independent $G(N,1/2)$ isomorphic is at most $N!2^{-\binom{N}{2}}=o(1)$ indicating a real difference in the behavior of the size of the largest common induced subgraph.

Indeed, in $2023$, Chatterjee and Diaconis proved in \cite{Diaco} that the typical size of the largest common induced subgraph of two independent uniform random graphs (equivalently two independent $G(N,1/2)$ Erd\H{o}s-Renyi graphs) is logarithmic. More precisely, there exists an explicit sequence of integers $(z_N)_{N\geq 1}$ such that as $N$ goes to $\infty$, with probability tending to one, the largest common induced subgraph has size $z_{N}$ or $z_{N}+1$. This phenomenon is called two-point concentration for the size of the largest common induced subgraph of two uniform random graphs.

Two-point concentration is a well-known behavior for several characteristics of random graphs. For example, for Erd\H{o}s-Renyi graphs $G(N,p)$:
\begin{itemize}
    \item for the clique number a two-point concentration phenomenon has been proved  by Matula \cite{Matula} when $p$ is constant;
    \item for the chromatic number Shamir and Spencer \cite{shamir1987sharp} proved a five-point concentration and Alon and Krivelevich \cite{alon} proved later a two-point concentration when $p=N^{-1/2-\delta}$ (see also \cite{achli, achli2});
    \item for the independence number, Bohman and Hofstad \cite{bohman2024two} proved a two-point concentration;
    \item for the domination number, Glebov and al. \cite{Dom1} proved a two-point concentration when $p\gg N^{-1/2}$ and Bohman and al. extended the result for $p\gg N^{-2/3}$.
\end{itemize} This phenomenon also occurs for other models of graphs, see for example \cite{Muller} for the chromatic number of a model of random geometric graph.

A natural question is to extend these two-point concentration results to hypergraphs, see \cite{dyer2015chromatic} for example for the chromatic number. A $d$-hypergraph\footnote{Also called $d$-uniform hypergraph in the litterature, however to avoid confusion with the probability notion of uniformity, we adopt this terminology.} $G=(V,E)$ is composed of a set of vertices $V$ and a set $E$ of hyperedges (\emph{i.e.} subsets of $V$) of size $d$. Hypergraphs appear naturally in a variety of contexts: to model satisfiability problems \cite{DBLP:conf/stoc/Marx10}, databases \cite{DBLP:journals/jacm/Fagin83}, recommendation systems in machine learning \cite{DBLP:conf/aaai/0013YYWC021} etc. For these reasons they are studied from both algorithmic and probabilistic points of view.

The aim of this paper is to extend the result of Chatterjee-Diaconis to uniform random $d$-hypergraphs of different sizes.

\begin{thm}\label{grosthm}
Let $(\indi(\inda))_{\inda\geq 1}$ be a sequence of positive integers such that, for all $\inda$, $\inda \geq \indi(\inda) $. Let $d\geq 3$ and $\Gamma_1$ and $\Gamma_2$ be two independent uniform random $d$-hypergraphs with respectively $\inda$ and $\indi(\inda)$ vertices. 
\begin{enumerate}

\item If $\liminf{\frac{\log_2(\indi(\inda))}{(\log_2\inda)^{1-\frac{1}{d-1}}}}>0$, 
 then w.h.p.\footnote{Throughout this paper we use the shortcut w.h.p. for sequences of events whose probability tends to one as $\inda$ tends to $+\infty$.} the largest common induced subgraph of $\Gamma_1$ and $\Gamma_2$ has size either $\lfloor x_{\inda}-(\log_2\inda)^{-1/d}\rfloor$ or  $\lfloor x_{\inda}+(\log_2\inda)^{-1/d}\rfloor$,
where\begin{equation}\label{defx}x_{\inda}=\left(d!\log_2(\inda\indi(\inda))\right)^{\frac{1}{d-1}} + \frac{d}{2}.\end{equation}
\item Assume on the contrary that $\log_2(\indi(\inda))=o\left((\log_2\inda)^{1-\frac{1}{d-1}}\right)$.\begin{enumerate} \item If for all $\inda$ large enough \begin{equation}\label{condi} \indi(\inda) >\bigg\lfloor \left(d! \log_2 \inda\right)^{\frac{1}{d-1}}+\frac{d}{2}-(\log_2\inda)^{-\frac{1}{d}}\bigg\rfloor,\end{equation} then there exists $(\mu_{\inda})_{\inda\geq 1}$ a positive sequence with limit $0$ such that w.h.p. the largest common induced subgraph of\, $\Gamma_1$ and $\Gamma_2$ has size either $\lfloor y_{\inda}-\mu_{\inda}\rfloor$ or  $\lfloor y_{\inda}+\mu_{\inda}\rfloor$,
where \begin{equation}\label{defy}y_{\inda}=\left(d!\log_2\inda\right)^{\frac{1}{d-1}} + \frac{d}{2}.\end{equation}
\item If for $\inda$ large enough, Eq. \eqref{condi} is not verified, then w.h.p. $\Gamma_2$ is an induced subgraph of $\Gamma_1$.
\end{enumerate}
\end{enumerate}
\end{thm}

Note that the condition of item 2b is almost sharp: if $$\indi(\inda)=\bigg\lfloor \left(d! \log_2 \inda\right)^{\frac{1}{d-1}}+\frac{d}{2}-(\log_2\inda)^{-\frac{1}{d}}\bigg\rfloor+2,$$ by item 2a, for $\inda$ large enough the size of the largest common induced subgraph is at most $y_{\inda}+\mu_{\inda}<y_{\inda}-(\log\inda)^{-\frac{1}{d}}+1\leq \indi(\inda)$, thus w.h.p. $\Gamma_2$ is not included in $\Gamma_1$.

The proof of Theorem \ref{grosthm} relies on the first and second moment methods. Even if our proofs are inspired by the ones used in the Chatterjee-Diaconis paper \cite{Diaco}, several differences appear. First, in their article they only deal with the cases corresponding to 1 for $\indi(\inda)=\inda$ and 2.b of Theorem \ref{grosthm}. Secondly, our asymptotic computations do not cover the case of uniform random graphs (\emph{i.e.} the case $d=2$). Indeed, if we take for all $\inda$, $\indi(\inda)=\inda$ and $d=2$ in the formula for $x_{\inda}$, we obtain $x_{\inda}=4\log_2\inda+1$ which does not coincide with the value of Chatterjee-Diaconis (we do not recover in our theorem the $\log_2\log_2$ term from \cite[Thm 1.1]{Diaco}). It was rather unexpected for us that the generalization to $d\geq 3$
relies on different asymptotics than the case $d=2$, this will become apparent in the calculations during the proof of Lemma 4.

\subsection{Notations and proof strategy}

We begin by defining the notion of induced subgraph:
\begin{dfn}
    For any $d$-hypergraph $G=(V,E)$, and any $k$-tuple $I=(i_1,\dots,i_k)$ of distinct elements of $V$, the subgraph $G_I$ of $G$ induced by $I$ is the subgraph whose set of vertices is $\{1,\dots, k\}$, and such that, for every $J\subset \{1,\dots,k\}$ of size $d$, $J$ is an hyperedge of $G_I$ if and only if $\{i_j\ |\ j\in J\}$ is an hyperedge of $G$.
\end{dfn}

Here are the different notations that are used throughout the article:
    \begin{itemize}
    \item The notation $\log$ stands for the natural logarithm in base $e$. Most of the proofs will use this logarithm;
    
        \item $d$ is an integer greater or equal to $3$;
        \item $(\indi(\inda))_{\inda\geq 1}$ is a sequence such that for all $\inda$, $\inda\geq \indi(\inda)$. For convenience, we will write $\indi$ instead of $\indi(\inda)$;
        \item $\MM=\sqrt{\inda\indi}$;
        
        \item $\Gamma_1$ and $\Gamma_2$ are two independent uniform random $d$-hypergraphs of respective size $\inda$ and $\indi$: each possible hyperedge is present independently with probability $1/2$;
        \item  $a = \frac{2}{\log 2}d!$;
        \item $(\beta_{\inda})_{\inda\geq 1}$ is a bounded sequence such that, if we set for every $\inda\geq 1$, \begin{equation}\label{defden}
    \nn_{\inda}:= (a\log \MM)^{\frac{1}{d-1}} + \beta_{\inda},
\end{equation}
then $\nn_{\inda}$ is a positive integer. We will write $\nn$ instead of $\nn_{\inda}$ to lighten the notations;
\item for every positive integers $M$ and $k$ with $k\leq M$, $\mathcal{A}_{M,k}$ is the set of $k$-tuples of distinct elements of $\{1,\dots,M\}$. For all element $F$ of $\mathcal{A}_{M,k}$ and all $i\in\{1,\dots k\}$ we denote by $f_i$ the $i$-th element of $F$;
\item for every positive integers $M$ and $k$ with $k\leq M$, recall that $|\mathcal{A}_{M,k}|=M(M-1)\dots(M-k+1)=(M)_k$;

\item the random variable $W$ is defined by $W=|\{(I,J)\in \mathcal{A}_{\inda,\nn}\times \mathcal{A}_{\indi,\nn}, \Gamma_{1,I}=\Gamma_{2,J}\}|$.
    \end{itemize}

With the above notations, having a common subgraph of size $\nn$ is thus equivalent to $W>0$. The proof of Theorem \ref{grosthm} consists in estimating the first and second moments of $W$ for well-chosen values of $\nn$ and using the first and second moment methods. 

In Section \ref{sec1}, the first moment of $W$ is computed explicitly, and we prove that asymptotically with high probability the largest common induced subgraph is of size at most $\lfloor x_{\inda}+o(1)\rfloor$ with the first moment method. 

In Section \ref{sec2}, we give an upper bound on the second moment of $W$ (which cannot be explicitly computed) to prove that asymptotically with high probability if $\ell\leq \indi(\inda)$ for $\inda$ large enough, the largest common induced subgraph is of size at least $\lfloor x_{\inda}-o(1)\rfloor$ with the second moment method. 

We conclude by carefully applying both results to the different cases of Theorem \ref{grosthm}.

\section{First moment}\label{sec1}

We will start with a technical lemma. Denote by $$U_{\alpha,\inda}:=2\log \MM-\frac{(\alpha \nn-1)\dots (\alpha \nn-(d-1))}{d!}\log 2$$
for $\alpha\in (0,1]$ fixed. Since the asymptotic expansion of $U_{\alpha, N}$ will be used several times in our estimations, we compute it separately.
\begin{lem}\label{calcul}
For every $\alpha \in(0,1]$, we have the following asymptotic expansion as $\inda$ goes to infinity:
    $$U_{\alpha,\inda}=(2-2\alpha^{d-1})\log \MM+\alpha^{d-2}\frac{d-1}{2}\frac{a^{\frac{d-2}{d-1}}\log 2}{d!}(d-2\alpha \beta_{\inda})(\log \MM)^{\frac{d-2}{d-1}}+\mathcal{O}((\log \MM)^{\frac{d-3}{d-1}}).$$
 For $\alpha=1$, we get

    $$U_{1,\inda}=\frac{d-1}{2}\frac{a^{\frac{d-2}{d-1}}\log 2}{d!}(d-2\beta_{\inda})(\log \MM)^{\frac{d-2}{d-1}}+\mathcal{O}((\log \MM)^{\frac{d-3}{d-1}}).$$
\end{lem}
\begin{proof}
Note that: 
\begin{align*}
U_{\alpha,\inda}&= 2\log \MM-\frac{\alpha^{d-1}\nn^{d-1}}{d!}\log{2}+{ d\choose 2}\frac{\alpha^{d-2}\nn^{d-2}}{d!}\log{2}+\mathcal{O}(\nn^{d-3})\\
&=2\log \MM-a\alpha^{d-1}\log \MM\frac{\log 2}{d!}-\alpha^{d-1}\frac{\log 2}{d!}(d-1)\beta_{\inda}a^{\frac{d-2}{d-1}}(\log \MM)^{\frac{d-2}{d-1}}\\
&\ \ \ \ \ \ \ \ \ +\alpha^{d-2}\frac{\log 2}{d!}{ d\choose 2}a^{\frac{d-2}{d-1}}(\log \MM)^{\frac{d-2}{d-1}}+\mathcal{O}\left((\log \MM)^{\frac{d-3}{d-1}}\right)
\end{align*}
by Eq. \eqref{defden} and since $(\beta_{\inda})_{\inda\geq 1}$ is bounded.
Thus, using the value of $a$, 
$$
U_{\alpha,\inda}=(2-2\alpha^{d-1})\log \MM+\alpha^{d-2}\frac{d-1}{2}\frac{a^{\frac{d-2}{d-1}}\log 2}{d!}(d-2\alpha \beta_{\inda})(\log \MM)^{\frac{d-2}{d-1}}+\mathcal{O}\left((\log \MM)^{\frac{d-3}{d-1}}\right)$$
concluding the computation.\end{proof}

The following lemma provides two asymptotic estimates for $\Es{W}$, both of which will be useful later.
\begin{lem}\label{lemesp}
    We have:\begin{align*} \bullet \text{ if $\indi<\ell$,\ \ \  \ }\Es{W}&=0\\ \bullet \text{ otherwise, }\ \Es{W}&=2^{-\binom{\nn}{d}} \MM^{2\nn}\frac{(\indi)_{\nn}}{\indi^{\nn}}(1+o(1))\\
    &=\exp\left((d-1)(d-2\beta_{\inda})\log \MM +\mathcal{O}\left((\log \MM)^{\frac{d-2}{d-1}}\right)\right).\end{align*}
\end{lem}

\begin{proof}
If $\indi<\ell$, then $W=0$ which implies the first equality. Otherwise assume that $\indi\geq \ell$. By linearity of expectation, 
\begin{align*}
\Es{W}&=|\mathcal{A}_{\inda,\nn}|\times|\mathcal{A}_{\indi,\nn}|\Pr{\Gamma_{1,(1,\dots,\nn)}=\Gamma_{2,(1,\dots,\nn)}}\\
&=\prod\limits_{i=0}^{\nn-1}\left(\inda-i\right)\prod\limits_{i=0}^{\nn-1}\left(\indi-i\right)2^{-\binom{\nn}{d}},
\end{align*}
thus $$\Es{W}=2^{-\binom{\nn}{d}}\MM^{2\nn}\prod\limits_{i=0}^{\nn-1}\left(1-\frac{i}{\inda}\right)\frac{(\indi)_{\nn}}{\indi^{\nn}}.$$
Since
$$\prod\limits_{i=0}^{\nn-1}\left(1-\frac{i}{\inda}\right)=\exp\left(\mathcal{O}\left(\sum\limits_{i=0}^{\nn-1}\frac{i}{\inda}\right)\right)=\exp\left(\mathcal{O}\left(\frac{\nn^2}{\inda}\right)\right)=\exp\left(o(1)\right),$$

we have $\Es{W}=(1+o(1))2^{-\binom{\nn}{d}}\MM^{2\nn}\frac{(\indi)_{\nn}}{\indi^{\nn}}$.

Moreover since $\indi\geq \nn$,
\begin{align*}1\geq \frac{(\indi)_{\nn}}{\indi^{\nn}}=\prod\limits_{i=0}^{\nn-1}\left(1-\frac{i}{\indi}\right)\geq \prod\limits_{i=0}^{\nn-1}\left(1-\frac{i}{\nn}\right)=\frac{\nn!}{\nn^{\nn}} =\exp\left(\mathcal{O}(\nn)\right),\end{align*}

therfore 
\begin{align}\label{eqbof}
    \frac{(\indi)_{\nn}}{\indi^{\nn}}&=\exp\left(\mathcal{O}(\nn)\right),
\end{align}

and
\begin{align*}
    \Es{W} &= \exp\left(2\nn \log \MM+ \mathcal{O}(\nn)-\frac{\nn(\nn-1)\dots(\nn-(d-1))}{d!}\log 2 \right) \\
&= \exp\left(\mathcal{O}(\nn)+\nn U_{1,\inda}\right). 
\end{align*}
Finally from Lemma \ref{calcul} we get:
\begin{align*}
    \Es{W} &= \exp\left(\mathcal{O}(\nn)+\left((a\log \MM)^{\frac{1}{d-1}}+\beta_{\inda}\right)\left(\frac{d-1}{2}\frac{a^{\frac{d-2}{d-1}}\log 2}{d!}(d-2\beta_{\inda})(\log \MM)^{\frac{d-2}{d-1}}+\mathcal{O}\left((\log \MM)^{\frac{d-3}{d-1}}\right)\right)\right)\\
&= \exp\left( \mathcal{O}((\log \MM)^{\frac{d-2}{d-1}})+(d-1)(d-2\beta_{\inda})\frac{a\log 2}{2d!}\log \MM\right) \\
&=\exp\left(\mathcal{O}((\log \MM)^{\frac{d-2}{d-1}})+(d-1)(d-2\beta_{\inda})\log \MM \right) \quad \text{as $a=\frac{2 d!}{\log 2}$}
\end{align*}concluding the proof. (Note that we used $d\geq 3$ when the $\mathcal{O}(\nn)=\mathcal{O}\left((\log \MM)^{\frac{1}{d-1}}\right)$ was absorbed by $\mathcal{O}\left((\log \MM)^{\frac{d-3}{d-1}}\right)$.)\end{proof}
With the first moment method we now deduce the following lemma.
\begin{lem}[Towards the upper bound for Theorem \ref{grosthm}]\label{eq2}
Let $(\varepsilon_{\inda})_{\inda\geq 1}$ be a bounded sequence such that $\varepsilon_{\inda}\gg (\log \MM)^{-\frac{1}{d-1}}$. Then w.h.p. the largest common induced subgraph of $\Gamma_1$ and $\Gamma_2$ has size less or equal to $\lfloor x_{\inda}+\varepsilon_{\inda}\rfloor$:
$$\Pr{\exists (I,J)\in \mathcal{A}_{\inda,\lfloor x_{\inda}+\varepsilon_{\inda}\rfloor+1}\times \mathcal{A}_{\indi(\inda),\lfloor x_{\inda}+\varepsilon_{\inda}\rfloor+1}\ \text{such that}\  \Gamma_{1,I}=\Gamma_{2,J}}\xrightarrow[\inda\to\infty]{} 0$$
where $(x_{\inda})_{\inda\geq 1}$ is defined in Eq. \eqref{defx}.\end{lem}

\begin{proof}
For $\inda\geq 2$, choose $$\beta_{\inda}=\lfloor (a\log \MM)^{\frac{1}{d-1}}+\frac{d}2+\varepsilon_{\inda}\rfloor+1-(a\log \MM)^{\frac{1}{d-1}}.$$ Since $\frac{d}2+\varepsilon_{\inda}<\beta_{\inda}\leq \frac{d}{2}+\varepsilon_{\inda}+1$, we get that $d-2\beta_{\inda}\leq -2\varepsilon_{\inda}$ and that $(\beta_{\inda})_{\inda\geq 1}$ is bounded. Moreover, $\nn=\lfloor x_{\inda}+\varepsilon_{\inda}\rfloor+1$.

We have for any $\inda$ such that $\indi(\inda)\geq \ell$,
\begin{align*}
\Pr{W>0}&\leq \Es{W}\\
&= \exp\left((d-1)(d-2\beta_{\inda})\log \MM +\mathcal{O}\left((\log \MM)^{\frac{d-2}{d-1}}\right)\right)\quad \text{by Lemma \ref{lemesp},}\\
&\leq \exp\left(-2\varepsilon_{\inda}(d-1)\log \MM +\mathcal{O}\left((\log \MM)^{\frac{d-2}{d-1}}\right)\right)\quad \text{as $d-2\beta_{\inda}\leq -2\varepsilon_{\inda}$.}
\end{align*}
As when $\indi(\inda)< \ell$, $\Pr{W>0}=0$, it holds for all $\inda$. Thus $\Pr{W>0}$ goes to $0$ as $\inda$ goes to infinity, which proves Lemma \ref{eq2}.\end{proof}

\section{Second moment}\label{sec2}
We now tackle the estimation of the second moment.
\begin{lem}
  There exist positive constants $K,K'$ which do not depend on $\inda$ such that  $$\Es{W^2}\leq (1+o(1))\Es{W}^2\left(1+K\exp\left(-K'(d-2\beta_{\inda})\log \MM+\mathcal{O}\left((\log \MM)^{\frac{d-2}{d-1}}\log(\log \MM)\right)\right)\right).$$
\end{lem}

\begin{proof}
Note that if $\ell>\indi$, $W=0$ thus $\Es{W^2}=0=\Es{W}$ which implies the lemma. Now assume that $\ell\leq \indi$.
First we expand the second moment:
\begin{align*}
\Es{W^2}&=\sum\limits_{(A,C)\in \mathcal{A}_{\indi,\nn}^2, (B,D)\in \mathcal{A}_{\inda,\nn}^2}\Pr{\Gamma_{2,A}=\Gamma_{1,B},\Gamma_{2,C}=\Gamma_{1,D}}.
\end{align*}
We now split the computation according to the size of $B\cap D$. If $B\cap D$ has size $m$, there exist $1\leq i_1<\dots<i_m\leq \nn$ and $m$ distinct integers $j_1,\dots, j_m\in\{1,\dots, \nn\}$ such that $b_{i_1}=d_{j_1},\dots, b_{i_m}=d_{j_m}$. Thus
\begin{align*}
\Es{W^2}&=S_0+\sum\limits_{m=1}^{\nn} S_m
\end{align*}
where $$S_0=\sum\limits_{(A,C)\in \mathcal{A}_{\indi,\nn}^2, (B,D)\in \mathcal{A}_{\inda,\nn}^2, B\cap D=\emptyset}\Pr{\Gamma_{2,A}=\Gamma_{1,B},\Gamma_{2,C}=\Gamma_{1,D}}$$
and for $m\in\{1,\dots, \nn\}$ 

$$S_m=\sum\limits_{(A,C)\in \mathcal{A}_{\indi,\nn}^2}\sum\limits_{1\leq i_1<\dots<i_m\leq \nn}\sum\limits_{(j_1,\dots,j_m)\in \mathcal{A}_{\nn,m}}\sum\limits_{\underset{\underset{ \forall p,\ b_{i_{p}}=d_{j_{p}}}{|B\cap D|=m}}{(B,D)\in \mathcal{A}_{\inda,\nn}^2}} \Pr{\Gamma_{2,A}=\Gamma_{1,B},\Gamma_{2,C}=\Gamma_{1,D}}.$$The quantity $S_0$ can be explicitly computed. If $(A,C)\in \mathcal{A}_{\indi,\nn}^2$ and $(B,D)\in \mathcal{A}_{\inda,\nn}^2$ are such that $B\cap D=\emptyset$, then by conditioning on $\Gamma_2$, 

$$\Pr{\Gamma_{2,A}=\Gamma_{1,B},\Gamma_{2,C}=\Gamma_{1,D}}=2^{-2\binom{\nn}{d}}.$$
Thus \begin{equation}\label{s0}S_0= (\inda)_{2\nn}(\indi)_{\nn}^22^{-2\binom{\nn}{d}}\leq \MM^{4\nn}\frac{(\indi)_{\nn}^2}{\indi^{2\nn}}2^{-2\binom{\nn}{d}}.\end{equation}

Now our aim is to simplify the expression of $S_m$ for $m\geq 1$. For every $m\geq 1$, every $(A,C)\in \mathcal{A}_{\indi,\nn}^2, (B,D)\in \mathcal{A}_{\inda,\nn}^2$, every $1\leq i_1<\dots<i_m\leq \nn$, every $(j_1,\dots,j_m)\in \mathcal{A}_{\nn,m}$ such that $|B\cap D|=m$ and for all $p \in\{1,\dots,m\}$, $b_{i_{p}}=d_{j_{p}}$, the event $\{\Gamma_{2,A}=\Gamma_{1,B},\Gamma_{2,C}=\Gamma_{1,D}\}$ can be written as the intersection of three events $\mathcal{E}_1, \mathcal{E}_2$ and $\mathcal{E}_3$ (see Fig. \ref{fig}).  
\begin{itemize}
    \item The first one is $\mathcal{E}_1=\{\Gamma_{2,A}=\Gamma_{1,B}\}$. This event forces all the $\binom{\nn}{d}$ hyperedges of $\Gamma_{1,B}$. Conditionally on $\Gamma_2$, $\mathcal{E}_1$ has probability $2^{-\binom{\nn}{d}}$. 

    \item  The second event is $\mathcal{E}_2=\{\Gamma_{2,(c_{j_1},\dots,c_{j_m})}=\Gamma_{2,(a_{i_1},\dots,a_{i_m})}\}$.  Moreover it is $\Gamma_2$-measurable.
        \item The last event $\mathcal{E}_3$ is that all the hyperedges whose vertices are not all in $\{j_1,\dots, j_{m}\}$ are the same in $\Gamma_{2,C}$ and $\Gamma_{1,D}$. Conditionally on $\Gamma_2$, $\mathcal{E}_3$ is independent of $\mathcal{E}_1$ (as no involved hyperedge of $\Gamma_1$ has only vertices in $B\cap D$) and has probability $2^{-\binom{\nn}{d}+\binom{m}{d}}$ since it involves only the hyperedges whose vertices are in $D$ but not all in $B\cap D$. 
    \end{itemize} 
    \begin{center}
\begin{figure}[htbp]
\includegraphics[scale=1]{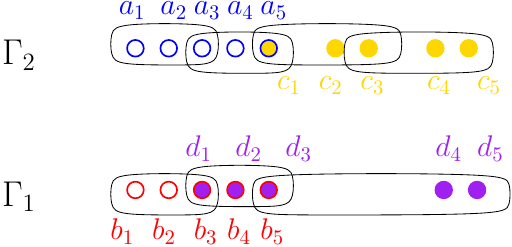}
\caption{An example of a $3$-hypergraph for which $\{\Gamma_{2,A}=\Gamma_{1,B},\Gamma_{2,C}=\Gamma_{1,D}\}$ occurs with $m=3$}\label{fig}
\end{figure}
\end{center}
    Note that $\mathcal{E}_1\cap \mathcal{E}_2=\mathcal{E}_1\cap \{\Gamma_{2,(c_{j_1},\dots,c_{j_m})}=\Gamma_{1,(d_{j_1},\dots,d_{j_m})}\}$ since $\mathcal{E}_1$ implies that $\Gamma_{1,(b_{i_1},\dots,b_{i_m})}=\Gamma_{2,(a_{i_1},\dots,a_{i_m})}$ and since $\Gamma_{1,(b_{i_1},\dots,b_{i_m})}=\Gamma_{1,(d_{j_1},\dots,d_{j_m})}$. So $\mathcal{E}_1\cap \mathcal{E}_2\cap \mathcal{E}_3=\{\Gamma_{2,A}=\Gamma_{1,B},\Gamma_{2,C}=\Gamma_{1,D}\}$ and 
    \begin{align*}
\Pr{\Gamma_{2,A}=\Gamma_{1,B},\Gamma_{2,C}=\Gamma_{1,D}}&=\Es{\Pr{\mathcal{E}_1\cap \mathcal{E}_2\cap \mathcal{E}_3|\Gamma_2}}\\
&=\Es{\mathbbm{1}_{\mathcal{E}_2}2^{-2\binom{\nn}{d}+\binom{m}{d}}}\\
&=2^{-2\binom{\nn}{d}+\binom{m}{d}}\Pr{\Gamma_{2,(a_{i_1},\dots,a_{i_m})}=\Gamma_{2,(c_{j_1},\dots,c_{j_m})}}.
\end{align*}
Therefore
$$\scalebox{0.8}{
$S_m=(\inda)_{\nn}(\inda-\nn)_{\nn-m}\sum\limits_{(A,C)\in\mathcal{A}_{\indi,\nn}^2}\sum\limits_{1\leq i_1<\dots<i_m\leq \nn}\sum\limits_{(j_1,\dots,j_m)\in \mathcal{A}_{\nn,m}}2^{-2\binom{\nn}{d}+\binom{m}{d}}\Pr{\Gamma_{2,(a_{i_1},\dots,a_{i_m})}=\Gamma_{2,(c_{j_1},\dots,c_{j_m})}}.$}$$

Fix $m\in \{1,\dots, \nn\}$, $A\in \mathcal{A}_{\indi,\nn}$, $1\leq i_1<\dots<i_m\leq \nn$ and 
$(j_1,\dots,j_m)\in \mathcal{A}_{\nn,m}$. Since $\Gamma_2$ is invariant by relabelling, we get:

$$\sum\limits_{C\in \mathcal{A}_{\indi,\nn}}\Pr{\Gamma_{2,(a_{i_1},\dots,a_{i_m})}=\Gamma_{2,(c_{i_1},\dots,c_{j_m})}}=\sum\limits_{C\in \mathcal{A}_{\indi,\nn}}\Pr{\Gamma_{2,(1,\dots,m)}=\Gamma_{2,(c_{i_1},\dots,c_{j_m})}}.$$

Thus 
\begin{align}
S_m&=(\indi)_{\nn}(\inda)_{\nn}(\inda-\nn)_{\nn-m}\binom{\nn}{m}\sum\limits_{C\in\mathcal{A}_{\indi,\nn}}\sum\limits_{(j_1,\dots,j_m)\in \mathcal{A}_{\nn,m}}2^{-2\binom{\nn}{d}+\binom{m}{d}}\Pr{\Gamma_{2,(1,\dots,m)}=\Gamma_{2,(c_{j_1},\dots,c_{j_m})}}\notag\\
&=(\indi)_{\nn}(\inda)_{\nn}(\inda-\nn)_{\nn-m}\binom{\nn}{m}(\indi-m)_{\nn-m}(\nn)_m\sum\limits_{(e_1,\dots,e_m)\in \mathcal{A}_{\indi,m}}2^{-2\binom{\nn}{d}+\binom{m}{d}}\Pr{\Gamma_{2,(1,\dots,m)}=\Gamma_{2,(e_1,\dots,e_m)}}\notag\\
&\leq \MM^{4\nn-2m}\nn^{2m}2^{-2\binom{\nn}{d}+\binom{m}{d}}\sum\limits_{(e_1,\dots,e_m)\in \mathcal{A}_{\indi,m}}\Pr{\Gamma_{2,(1,\dots,m)}=\Gamma_{2,(e_1,\dots,e_m)}}\label{sm}.
\end{align}

Fix $1\leq m \leq \nn$, set $T_m= \sum\limits_{(e_1,\dots,e_m)\in \mathcal{A}_{\indi,m}}\Pr{\Gamma_{2,(1,\dots,m)}=\Gamma_{2,(e_1,\dots,e_m)}}$. Using the same techniques as the ones used since the beginning of the proof (splitting according to the intersection as well as the independence of the edges, and, for the last equality, the invariance by relabelling), we obtain the following equalities:
\begin{align*}
    T_m&=\sum\limits_{j=0}^m\sum\limits_{{1\leq p_1<\dots<p_j\leq m}}\sum\limits_{\underset{\forall i,\ e_{p_i}\in\{1,\dots,m\}}{\underset{j=|\{1,\dots,m\}\cap\{e_1,\dots,e_m\}|}{(e_1,\dots,e_m)\in \mathcal{A}_{\indi,m}}}}\Pr{\Gamma_{2,(1,\dots,m)}=\Gamma_{2,(e_1,\dots,e_m)}}\\
    &=\sum\limits_{j=0}^m\sum\limits_{{1\leq p_1<\dots<p_j\leq m}}2^{-\binom{m}{d}+\binom{j}{d}}\sum\limits_{\underset{\forall i,\ e_{p_i}\in\{1,\dots,m\}}{\underset{j=|\{1,\dots,m\}\cap\{e_1,\dots,e_m\}|}{(e_1,\dots,e_m)\in \mathcal{A}_{\indi,m}}}}\Pr{\Gamma_{2,(p_1,\dots,p_{j})}=\Gamma_{2,(e_{p_1},\dots,e_{p_{j}})}}\\
     &=\sum\limits_{j=0}^m(\indi-m)_{m-j}2^{-\binom{m}{d}+\binom{j}{d}}\sum\limits_{1\leq p_1<\dots<p_j\leq m}\sum\limits_{(f_1,\dots,f_j)\in \mathcal{A}_{m,j}}\Pr{\Gamma_{2,(p_1,\dots,p_j)}=\Gamma_{2,(f_1,\dots,f_{j})}}\\
     &=\sum\limits_{j=0}^m(\indi-m)_{m-j}\binom{m}{j}2^{-\binom{m}{d}+\binom{j}{d}}\sum\limits_{(f_1,\dots,f_j)\in \mathcal{A}_{m,j}}\Pr{\Gamma_{2,(1,\dots,j)}=\Gamma_{2,(f_1,\dots,f_{j})}}.\end{align*} Therefore \begin{equation}\label{TNm}T_m\leq \sum\limits_{j=0}^m\indi^{m-j}m^j2^{-\binom{m}{d}+\binom{j}{d}}m^j.\end{equation}

Throughout the next computations, $K,K',K_1,K_2, K_3$ will be arbitrary positive constants which do not depend on $\inda$. Combining Eqs. \eqref{s0}, \eqref{sm} and \eqref{TNm}, we get
\begin{align*}
\Es{W^2}&\leq \MM^{4\nn}\frac{(\indi)_{\nn}^2}{\indi^{2\nn}}2^{-2\binom{\nn}{d}}+\sum\limits_{m=1}^{\nn}\MM^{4\nn-2m}\nn^{2m}2^{-2\binom{\nn}{d}+\binom{m}{d}}\sum\limits_{j=0}^m\indi^{m-j}m^j2^{-\binom{m}{d}+\binom{j}{d}}m^j\\
&\leq \MM^{4\nn}2^{-2\binom{\nn}{d}}\frac{(\indi)_{\nn}^2}{\indi^{2\nn}}\left(1+\frac{\indi^{2\nn}}{(\indi)_{\nn}^2}\sum\limits_{j=0}^{\nn}\sum\limits_{m=\max(j,1)}^{\nn}(\inda\indi)^{-m}\nn^{2m}2^{\binom{j}{d}}\indi^{m-j}\nn^{2j}\right)\\
&\leq \MM^{4\nn}2^{-2\binom{\nn}{d}}\frac{(\indi)_{\nn}^2}{\indi^{2\nn}}\left(1+\frac{\indi^{2\nn}}{(\indi)_{\nn}^2}\sum\limits_{m=1}^{\nn}\left(\frac{\nn^2}{\inda}\right)^m+\frac{\indi^{2\nn}}{(\indi)_{\nn}^2}\sum\limits_{j=1}^{\nn}\sum\limits_{m=j}^{\nn}\left(\frac{\nn^2}{\inda}\right)^m2^{\binom{j}{d}}\indi^{-j}\nn^{2j}\right)\\
&\leq  \MM^{4\nn}2^{-2\binom{\nn}{d}}\frac{(\indi)_{\nn}^2}{\indi^{2\nn}}\left(1+K\frac{\indi^{2\nn}}{(\indi)_{\nn}^2}\frac{\nn^2}{\inda}+K\frac{\indi^{2\nn}}{(\indi)_{\nn}^2}\sum\limits_{j=1}^{\nn}\left(\frac{\nn^2}{\inda}\right)^j2^{\binom{j}{d}}\indi^{-j}\nn^{2j}\right)\\
&\leq  \MM^{4\nn}2^{-2\binom{\nn}{d}}\frac{(\indi)_{\nn}^2}{\indi^{2\nn}}\left(1+K\frac{\indi^{2\nn}}{(\indi)_{\nn}^2}\frac{\nn^2}{\inda}+K\frac{\indi^{2\nn}}{(\indi)_{\nn}^2}\sum\limits_{j=1}^{\nn}\nn^{4j}\MM^{-2j}2^{\binom{j}{d}}\right).
\end{align*}
The constant $K$ in the second last line comes from a comparison with the sum of a geometric series since $\nn^2\ll \inda$. 

By Eq. \eqref{eqbof}, $\frac{\indi^{2\nn}}{(\indi)_{\nn}^2}\frac{\nn^2}{\inda}=\exp(\mathcal{O}(\nn)-\log(\inda))=o(1)$. Thus, by Lemma \ref{lemesp} for $\inda$ large enough:  
\begin{align*}
\Es{W^2}&\leq  \MM^{4\nn}2^{-2\binom{\nn}{d}}\frac{(\indi)_{\nn}^2}{\indi^{2\nn}}\left(1+o(1)+K\frac{\indi^{2\nn}}{(\indi)_{\nn}^2}\sum\limits_{j=d}^{\nn}\nn^{4\nn}\MM^{-2j}2^{\binom{j}{d}}\right)\\
&\leq (1+o(1))\Es{W}^2\left(1+K\frac{\indi^{2\nn}}{(\indi)_{\nn}^2}\sum\limits_{j=d}^{ n}\nn^{4\nn}\MM^{-2j}2^{\binom{j}{d}}\right).
\end{align*}
Now fix $\alpha \in (0,1)$. Since $\nn\log(\nn)=o(\log \MM)$ and $x\mapsto -\left(2\log \MM-\frac{(x-1)\dots (x-(d-1))\log{2}}{d!}\right) $ is an increasing function on $[d,+\infty[$ we have

\begin{align*}
\frac{\indi^{2\nn}}{(\indi)_{\nn}^2}\sum\limits_{d\leq j \leq \alpha \nn}\nn^{4\nn}\MM^{-2j}2^{\binom{j}{d}}&= \frac{\indi^{2\nn}}{(\indi)_{\nn}^2}\sum\limits_{d\leq j \leq \alpha \nn}\exp\left(-j\left(2\log \MM-\frac{(j-1)\dots (j-(d-1))\log{2}}{d!}\right)+4\nn\log \nn\right)\\
&= \frac{\indi^{2\nn}}{(\indi)_{\nn}^2}\sum\limits_{d\leq j \leq \alpha \nn}\exp\left(-jU_{\alpha,\inda}+o(\log \MM)\right)\\
&\leq \frac{\indi^{2\nn}}{(\indi)_{\nn}^2}\sum\limits_{d\leq j \leq \alpha \nn}\exp\left(-K_1j\log \MM\right)\quad \text{by Lemma \ref{calcul}}\\
&\leq K_3\exp(-K_2\log \MM+\mathcal{O}(\nn))=o(1).
\end{align*}
The second last line is a consequence of Lemma \ref{calcul} as $U_{\alpha, \inda}\geq K_1\log(\MM)$ for $\inda$ large enough. And the last inequality is true from a comparison with the sum of a geometric series and by Eq. \eqref{eqbof}.

We also have, since $\nn\log(\nn)=\mathcal{O}\left((\log \MM)^{\frac{1}{d-1}}\log(\log \MM)\right)$

\begin{align*}
    \frac{\indi^{2\nn}}{(\indi)_{\nn}^2}\sum\limits_{\alpha \nn<j\leq\nn}\nn^{4\nn}\MM^{-2j}2^{\binom{j}{d}}&=\frac{\indi^{2\nn}}{(\indi)_{\nn}^2} \sum\limits_{\alpha \nn<j\leq\nn}\exp\left(-j\left(2\log \MM-\frac{(j-1)\dots (j-(d-1))\log(2)}{d!}\right)+4\nn\log \nn\right)\\
&= \frac{\indi^{2\nn}}{(\indi)_{\nn}^2}\sum\limits_{\alpha \nn<j\leq\nn}\exp\left(-jU_{1,\inda}+\mathcal{O}\left((\log \MM)^{\frac{1}{d-1}}\log(\log \MM)\right)\right)\\
&\leq \frac{\indi^{2\nn}}{(\indi)_{\nn}^2}\nn\exp\left(-K'(d-2\beta_{\inda})\log \MM+\mathcal{O}\left((\log \MM)^{\frac{d-2}{d-1}}\log(\log \MM)\right)\right)\\
&\leq\exp\left(-K'(d-2\beta_{\inda})\log \MM+\mathcal{O}\left((\log \MM)^{\frac{d-2}{d-1}}\log(\log \MM)\right)\right).
\end{align*}
In the last asymptotic expansion, we used Eq. \eqref{eqbof}. Note that here the assumption $d\geq 3$ is crucial: the $\mathcal{O}\left((\log \MM)^{\frac{d-2}{d-1}}\log(\log \MM)\right)$ term enables us to take care of both the error term in $U_{1,\inda}$, and the $\mathcal{O}\left((\log \MM)^{\frac{1}{d-1}}\log(\log \MM)\right)$ if $d=3$.
Finally summing the two previous inequalities gives
\begin{align*}
\Es{W^2}&\leq (1+o(1))\Es{W}^2\left(1+K\exp\left(-K'(d-2\beta_{\inda})\log \MM+\mathcal{O}\left((\log \MM)^{\frac{d-2}{d-1}}\log(\log \MM)\right)\right)\right)
\end{align*}concluding the proof.\end{proof}

 With this upper bound on the second moment, we prove the following lemma:
 \begin{lem}[Towards a lower bound for Theorem \ref{grosthm}]\label{eq1}
  Let $(\varepsilon_{\inda})_{\inda\geq 1}$ be a bounded sequence such that $\varepsilon_{\inda}\gg\frac{\log\log\MM}{(\log \MM)^{\frac{1}{d-1}}}$. Assume that for $\inda$ large enough, $\indi \geq \lfloor x_{\inda}-\varepsilon_{\inda}\rfloor$. Then w.h.p. largest common induced subgraph of $\Gamma_1$ and $\Gamma_2$ has size greater or equal to $\lfloor x_{\inda}-\varepsilon_{\inda}\rfloor$: 
 $$\Pr{\exists (I,J)\in \mathcal{A}_{\inda,\lfloor x_{\inda}-\varepsilon_{\inda}\rfloor}\times \mathcal{A}_{\indi(\inda),\lfloor x_{\inda}-\varepsilon_{\inda}\rfloor} \ \text{such that}\   \Gamma_{1,I}=\Gamma_{2,J}}\xrightarrow[\inda\to\infty]{} 1$$
 where $(x_{\inda})_{\inda\geq 1}$ is defined in Eq. \eqref{defx}.
\end{lem}

\begin{proof}
Choose $(\beta_{\inda})_{\inda\geq 1}$ such that for every large $\inda$, $$\beta_{\inda}=\lfloor (a\log \MM)^{\frac{1}{d-1}}+\frac{d}2-\varepsilon_{\inda}\rfloor-(a\log \MM)^{\frac{1}{d-1}}.$$ Since $\frac{d}2-1-\varepsilon_{\inda}<\beta_{\inda}\leq \frac{d}{2}-\varepsilon_{\inda}$, which implies that $d-2\beta_{\inda}\geq 2\varepsilon_{\inda}$ and that $(\beta_{\inda})_{\inda\geq 1}$ is bounded. Moreover $\nn=\lfloor x_{\inda}-\varepsilon_{\inda}\rfloor$.

Therefore 
\begin{align*}
\Es{W^2}&\leq (1+o(1))\Es{W}^2\left(1+K\exp\left(-K'\varepsilon_{\inda}\log \MM+\mathcal{O}\left((\log \MM)^{\frac{d-2}{d-1}}\log(\log \MM)\right)\right)\right)\\
&\leq (1+o(1))\Es{W}^2\left(1+K\exp\left(-K'(1+o(1))\varepsilon_{\inda}\log \MM\right)\right)\quad \text{as $1-\frac{1}{d}>\frac{d-2}{d-1}$}\\
&\leq (1+o(1))\Es{W}^2\left(1+o(1)\right)\\
&\leq (1+o(1))\Es{W}^2.
\end{align*}

Thus since $\Pr{W>0}\geq \frac{\Es{W}^2}{\Es{W^2}}\geq 1+o(1)$, $\Pr{W>0}$ goes to $1$ as $\inda$ goes to infinity, which proves Lemma \ref{eq1}.\end{proof}

\begin{proof}[Proof of Theorem \ref{grosthm}]

First, we prove item 1 of Theorem \ref{grosthm}. For $\inda$ large enough:
\begin{itemize}
    \item since $\liminf{\frac{\log_2(\indi(\inda))}{(\log_2\inda)^{1-\frac{1}{d-1}}}}>0$, $\indi\geq \lfloor x_{\inda}-(\log_2\inda)^{-1/d}\rfloor$;
    \item $0\leq \lfloor x_{\inda}+(\log_2\inda)^{-1/d}\rfloor-\lfloor x_{\inda}-(\log_2\inda)^{-1/d}\rfloor\leq 1$ since $(\log_2\inda)^{-1/d}$ tends to $0$.
\end{itemize} Since $(\log_2\inda)^{-1/d}=\Theta\left((\log_2\MM)^{-1/d}
\right)\gg \frac{\log\log\MM}{(\log\MM)^{\frac{1}{d-1}}}$, item 1 of Theorem \ref{grosthm} is a consequence of Lemmas \ref{eq1} and \ref{eq2} applied with $\varepsilon_{\inda}=(\log_2\inda)^{-1/d}$.

We now proceed to the proof of items 2a and 2b of Theorem \ref{grosthm}. Assume that for all $\inda$ large enough \begin{equation*} \indi(\inda) \geq \bigg\lfloor \left(d! \log_2 \inda\right)^{\frac{1}{d-1}}+\frac{d}{2}-(\log_2\inda)^{-\frac{1}{d}}\bigg\rfloor\end{equation*}
and that $\log_2\indi=o\left((\log_2\inda)^{1-\frac{1}{d-1}}\right)$.
Note that 
\begin{align*}(d!\log_2(\inda\indi))^{\frac{1}{d-1}}&=\left(d!\log_2\inda\right)^{\frac{1}{d-1}}+\left(d!\log_2\inda\right)^{\frac{1}{d-1}}\left(\left(1+\frac{\log_2\indi}{\log_2\inda}\right)^{\frac{1}{d-1}}-1\right)\\
&=\left(d!\log_2\inda\right)^{\frac{1}{d-1}}+\mathcal{O}\left((\log_2\inda)^{\frac{1}{d-1}}\frac{\log_2\indi}{\log_2\inda}\right)\\
&=\left(d!\log_2\inda\right)^{\frac{1}{d-1}}+o\left(1\right)\quad \text{as $\log_2\indi=o\left((\log_2\inda)^{1-\frac{1}{d-1}}\right)$}.
\end{align*}
Take $(\mu_{\inda})_{\inda\geq 1}$ a sequence such that  $\mu_{\inda}\geq (\log_2 \inda)^{-1/d}$ for all $\inda\geq 2$ and $$1\gg \mu_{\inda}\gg \max\left(\left\vert \left(d!\log_2\inda\right)^{\frac{1}{d-1}}-(d!\log_2(\inda\indi))^{\frac{1}{d-1}}\right\vert, \frac{\log\log\MM}{(\log \MM)^{\frac{1}{d-1}}}\right).$$ Applying Lemma \ref{eq2} with $\varepsilon_{\inda}=\mu_{\inda}+ \left(d!\log_2\inda\right)^{\frac{1}{d-1}}-(d!\log_2(\inda\indi))^{\frac{1}{d-1}}$ and Lemma~\ref{eq1} with $\varepsilon_{\inda}=\mu_{\inda}-\left(d!\log_2\inda\right)^{\frac{1}{d-1}}+(d!\log_2(\inda\indi))^{\frac{1}{d-1}}$ implies item 2a of Theorem \ref{grosthm}.

For item 2b, note that if for $\inda$ large enough $$\indi(\inda)=\bigg\lfloor \left(d! \log_2 \inda\right)^{\frac{1}{d-1}}+\frac{d}{2}-(\log_2\inda)^{-\frac{1}{d}}\bigg\rfloor=\lfloor y_{\inda}-(\log_2\inda)^{-\frac{1}{d}}\rfloor,$$ the previous argument with $y_{\inda}=(\log \inda)^{-1/d}$ implies that w.h.p. the largest common induced subgraph has size at least $\lfloor y_{\inda}-(\log_2\inda)^{-1/d}\rfloor$. Thus it must be $\Gamma_2$: w.h.p. $\Gamma_2$ is an induced subgraph of $\Gamma_1$.

\noindent The general case for item 2b is a simple argument of monotonicity. If for $\inda$ large enough $\indi(\inda)\leq \bigg\lfloor \left(d! \log_2 \inda\right)^{\frac{1}{d-1}}+\frac{d}{2}-(\log_2\inda)^{-\frac{1}{d}}\bigg\rfloor$, we can complete $\Gamma_2$ in a uniform random $d$-hypergraph of size $\lfloor y_{\inda}-(\log_2\inda)^{-1/d}\rfloor$. Since for $\inda$ large enough w.h.p. this larger uniform $d$-hypergraph is an induced subgraph in $\Gamma_1$, the same goes for $\Gamma_2$, concluding the proof.\end{proof}

\section{Remaining questions}\label{conclusion}

\begin{itemize}

\item Theorem \ref{grosthm} implies, in the case of item 1 and 2a, that for many $\inda$'s, there is actually a one-point concentration: \emph{e.g.} for item 1 there are arbitrary large sequences of consecutive $\inda$ such that $\lfloor x_{\inda}+(\log_2 \inda)^{-1/d}\rfloor=\lfloor x_{\inda}-(\log_2 \inda)^{-1/d}\rfloor$. However, there are also arbitrary large sequences of consecutive $\inda$ such that $\lfloor x_{\inda}+(\log_2 \inda)^{-1/d}\rfloor= \lfloor x_{\inda}-(\log_2 \inda)^{-1/d}\rfloor$+1. In this case, it is still open to know if there is concentration in only one of these values. We would need more precise asymptotic estimates for the second moment, or a new method to answer this question.

\item Very recently in \cite{Diacogénéralisé}, Surya, Warnke and Zhu provided a generalization of the result of Chatterjee and Diaconis to  Erd\H{o}s-Renyi graphs of parameters different than $\frac12$. They show that for $p,q\in(0,1)$ the size of the largest common induced subgraph of a $G(N,q)$ and a $G(N,p)$ Erd\H{o}s-Renyi graph is logarithmic, and that there is also a two-point concentration phenomenon. This extension is however a lot more involved: to avoid the explosion of the second moment, they consider only pseudorandom graphs in the computation of the second moment.

Investigate whether the approach employed in \cite{Diacogénéralisé} can be used to prove that there is a two-point concentration phenomenon for analogous models of random hypergraphs of different size is a natural extension of the present paper.

\end{itemize}

\textbf{Acknowledgements.} This work benefited from earlier discussions with Tadas Temcinas, Eva Maria-Hainzl, Matias Pavez-Signé during the RandNET summer school in Eindhoven in 2022. I would also like to thank Lucas Gerin and Frédérique Bassino for useful discussions and for carefully reading many earlier versions of this manuscript.

\bibliographystyle{plain}
\bibliography{biblio}

\vfill 
\noindent \textsc{Théo Lenoir} \verb|theo.lenoir@polytechnique.edu|\\
\textsc{Cmap, Cnrs}, \'Ecole polytechnique,\\
Institut Polytechnique de Paris,\\
91120 Palaiseau, France

\end{document}